\documentclass[11pt]{article}
\usepackage{amssymb}
\usepackage[top=3.25cm,bottom=3.25cm,left=2cm,right=2cm]{geometry}
\usepackage{graphicx}
\usepackage{amsmath,amsthm,amssymb}
\usepackage{enumerate}
\usepackage{enumerate}
\usepackage{graphicx}
\usepackage{amsmath}
\usepackage{amsthm}
\usepackage{amssymb}
\usepackage{indentfirst}
\usepackage[mathscr]{eucal}
\usepackage{mathrsfs}
\usepackage{graphicx}

\usepackage{color}

\usepackage{ifpdf}
\usepackage{makeidx}

\usepackage{color}

\usepackage{ifpdf}
\usepackage{makeidx}

\def\a{\alpha}
\def\b{\beta}

\def\om{\omega}

\def\wt{\widetilde}

\def\i{\indent}

\def\Res{\mathop{\rm Res}}

\def\wt{\widetilde}
\def\ft{\widetilde f}
\usepackage{enumerate}

\usepackage[mathscr]{eucal}
\usepackage{mathrsfs}

\newtheorem{theorem}{Theorem}[section]
\newtheorem{defi}[theorem]{Definition}
\newtheorem{lemma}[theorem]{Lemma}
\newtheorem{prop}[theorem]{Proposition}

\numberwithin{equation}{section}

\def\Ai{\text{Ai}}
\def\Bi{\text{Bi}}
\def\A{\text{A}}
\def\pt{\widetilde{\phi}}
\def\po{\phi^*}
\def\({\left(}
\def\){\right)}
\def\hh{\notag \\ &=}

\begin{document}
\author{Y. Lin$^{\,a,b}$ and R. Wong$^{\,b}$
\\
\hspace{.7cm}
 \hbox{\small \emph{ $^a$ Department of Mathematics, Zhongshan University, Guangzhou,
China}}
\\
 \hbox{\small
\emph{$^b$ Department of Mathematics, City University of Hong Kong, Tat Chee Avenue, Kowloon,
Hong Kong}
}
}
\title{Global Asymptotics of the Discrete Chebyshev Polynomials}
\date{}
\maketitle

\begin{abstract}
In this paper, we study the asymptotics of the discrete Chebyshev polynomials
$t_{n}(z,N)$ as the degree grows to infinity.
Global asymptotic formulas are obtained as $n\rightarrow\infty$, when the
ratio of the parameters $n/N=c$ is a constant in the interval $(0,1)$.
Our method is based on a modified version of
the Riemann-Hilbert approach first introduced by Deift and Zhou.
\end{abstract}

\section{Introduction}
The discrete Chebyshev polynomials were introduced by Chebyshev
in the study of least-squares data fitting.
They are  explicitly given by
\begin{equation}
    t_{n}(z,N)
=
    n!\Delta^{n} \left\{  {z \choose n} {z-N \choose n}
    \right\},\qquad n=0,1,\cdots,N-1
,
\label{t}
\end{equation}
where $\Delta$ is the forward difference operator with unit spacing
on the variable $z$; see  \cite[(2.8.1)]{Szego}.
These polynomials
are orthogonal on the discrete set $\{ 0,1,\cdots,N-1 \}$
with respect to the weight function $w(k)=1$, $k=0,1\cdots N-1$. More precisely, we have
\begin{equation}
    \sum_{k=0}^{N-1}t_{n}(k,N)t_{m}(k,N)
=
    \frac{ (N+n)! }{ (2n+1)(N-n-1)!}
    \delta_{n,m}
,
\qquad
    n,m=0,1,\cdots,N-1
.
\end{equation}
Like all other classical
orthogonal polynomials,
these polynomials also satisfy a
three-term recurrence relation. Indeed, we have
\begin{equation}
    (n+1)t_{n+1}(z,N)
=
    2(2n+1)(z-\tfrac12 (N-1) )
   t_{n}(z,N)
-
    n(N^2-n^2)
    t_{n-1}(z,N)
;
\label{recurrence relations}
\end{equation}
see Gautschi \cite{Gautschi}.
\\

The discrete Chebyshev polynomials are also known as
a special case of the Hahn polynomials \cite[p.174]{BealsWong}
\begin{equation}
  Q_n(z;\a,\b,N):={}_3F_2(-n,-z,n+\a+\b+1;-N,\a+1;1);
\end{equation}
see also \cite{KarlinMcgregor}.
Recall the weight function of the Hahn polynomials
\begin{equation}
   w(z;\a,\b,N):=
  \frac{ (\a+1)_z }{ z! } \frac{ (\b+1)_{N-1-z} }{ (N-1-z)! },
\qquad
z=0,1,\cdots,N-1.
\end{equation}
Note that by taking
$\a=\b=0$, the weight function becomes 1.
In this case, the Hahn polynomials reduce to the discrete Chebyshev polynomials except for a constant factor.


For further properties and applications, we refer to Szeg\"o \cite{Szego} and Hildebrand \cite{Hildebrand}.
Like other discrete orthogonal polynomials,
these polynomials
do not satisfy a second-order linear differential equation.
Hence,
the asymptotic theory for differential equations
can not be applied.
By using the steepest descent method
for double integrals,
Pan and Wong \cite{PanWong} have recently derived infinite asymptotic expansions for
the discrete Chebyshev polynomials when the variable $z$ is real  and the ratio $n/N$ lies in the interval $(0,1)$;
their results involve a confluent hypergeometric function and its derivative.
For more information about the integral methods,
we refer to Wong \cite{WongBook}.
On the other hand,
Baik et al. \cite{BaikBook} have studied the asymptotics of
discrete orthogonal polynomials with respect to a general weight function by using
the Riemann-Hilbert approach. The results in \cite{BaikBook} are very general, but
it is difficult to use them to write out explicit formulas for specific polynomials.
In the case of the Hahn polynomials (1.4), only the situation when both parameters
$\a$ and $\b$ are large has been considered. More precisely, the authors of \cite{BaikBook}
considered the case when $\a=NA$ and $\b=NB$, where $A$ and $B$ are fixed positive
numbers, thus excluding the special case of the discrete Chebyshev polynomials (i.e., $\a=\b=0$).
Moreover, the results in \cite{BaikBook} are more local in nature; that is, more formulas are
needed to describe the behavior of the orthogonal polynomials in the complex plane.

In this paper,
we investigate the asymptotics
of the discrete Chebyshev polynomials as $n\rightarrow \infty$,
when the
ratio of the parameters $n/N$ is a constant $c\in(0,1)$.
Global asymptotic formulas are obtained in the
complex $z$-plane when $n$ goes to infinity.
Our approach is based on a modified version of the steepest-descent method
for Riemann-Hilbert problems
introduced by Deift and Zhou
in \cite{Deift}; see \cite{DaiWong}, \cite{OuWong} and
\cite{X.S.Wang}.
More precisely,
we derive two Airy-type asymptotic formulas
for the discrete Chebyshev polynomials.
The regions of validity of these formulas overlap,
and together they cover the whole complex plane.

The presentation of this paper is arranged as follows:
In Section 2, we first present the basic Riemann-Hilbert
problem associated with the discrete Chebyshev
polynomials. Then, we transform the basic Riemann-Hilbert
problem into
a continuous one, which is similar to that in \cite{DaiWong,X.S.Wang}.
In Section 3, we introduce some auxiliary functions,
known as the $g$-function and the $\phi$-function.
Moreover, we define a function $D(z)$,
which is analogous to the function first used in \cite{X.S.Wang}
to remove the jumps of the continuous Riemann-Hilbert
problem near the endpoints of the interval of orthogonality.
In Section 4, we construct our parametrix by using
Airy functions and their derivatives, and
prove that this parametrix is asymptotically
equal to
the solution of the
Riemann-Hilbert
problem associated with  the discrete Chebyshev
polynomials formulated in Section 2.
Our main result is given in Section 5.
In Section 6, we compare our formulas with those given by
Pan and Wong \cite{PanWong}.

\section{Riemann-Hilbert Problem}

We start with the fundamental Riemann-Hilbert problem (RHP) for the discrete orthogonal polynomials.
Note that the leading coefficient of $t_{n}(z,N)$ is $(2n)!/n!^2$. Hence,
the monic discrete Chebyshev polynomials are given by
\begin{equation}
  \pi_{N,n}(z):=\frac{n!^2}{(2n)!} t_{n}(z,N)
.
\label{pi_definition}
\end{equation}
Clearly, they satisfy the new orthogonality relation
\begin{equation}
     \sum_{k=0}^{N-1}\pi_{N,n}(k)\pi_{N,m}(k)
=
    \delta_{n,m}/\gamma_{N,n}^2
,
\end{equation}
where
$$
    \gamma_{N,n}^2
    =
    \frac{ (2n+1)\Gamma(N-n)\Gamma^2(2n+1) }{ \Gamma(N+n+1)\Gamma^4(n+1) }
.
$$
Consider the following RHP for a $2\times2$ matrix-value function $Y(z)$:

\begin{enumerate}
\item[($\,Y_a$)] $Y(z)$ is analytic for $z\in \mathbb{C}\setminus \mathbb{N}$;
\item[($\,Y_b$)] at each $z=k\in \mathbb{N}$, the first column of $Y$ is analytic and the second column of $Y$ has a simple pole where the residue is
$$
    \Res\limits_{z=k} Y(z)=\lim_{z\rightarrow k} Y(z)
    \begin{pmatrix} 0 & 1  \\ 0 & 0 \end{pmatrix};
$$
    \item[($\,Y_c$)] as $z\rightarrow \infty$,
$$
    Y(z)
=
    \left(  I + O\left(\frac1z \right)  \right)
    \begin{pmatrix}
    z^n  & 0\\
    0 & z^{-n}
    \end{pmatrix}.
$$
\end{enumerate}
By a well-known theorem of Fokas, Its and Kitaev \cite{FokasIts}
(see also Baik et al. \cite{BaikBook}),
we have
\begin{theorem}
The unique solution to the above RHP is given by
\begin{equation}
    Y(z)
:=
    \begin{pmatrix}
    \pi_{N,n}(z)  &  \sum\limits_{k=0}^{N-1}\dfrac{ \pi_{N,n}(k) }{z-k}\\
    \\
    \gamma_{N,n-1}^2\pi_{N,n-1}(z)  &  \sum\limits_{k=0}^{N-1}
    \dfrac{\gamma_{N,n-1}^2\pi_{N,n-1}(k) }{ z-k }
    \end{pmatrix}
.
\label{Y}
\end{equation}
\end{theorem}

Let $X_N$ denote the set defined by
$$
    X_N:=\{x_{N,k} \}_{k=0}^{N-1}, \qquad \text{where\ \ } x_{N,k}:=\frac{k+1/2}{N}.
$$
The $x_{N,k}$'s are called \emph{nodes} and they all lie in the interval $(0,1)$.
Following Baik et al. \cite{BaikBook},
we introduce the first transformation
\begin{align}
    H(z)
&:=
    \begin{pmatrix}
     N^{-n}  & 0 \\
    0 & N^n
    \end{pmatrix}
    \,  Y(Nz-1/2)
        \begin{pmatrix}
         \prod\limits_{j=0}^{N-1}(z-x_{N,j})^{-1} & 0 \\
        0 & \prod\limits_{j=0}^{N-1}(z-x_{N,j})
        \end{pmatrix}
\nonumber
\\
&\ =
    N^{-n\sigma_3}\, Y(Nz-1/2)\left[
    \prod\limits_{j=0}^{N-1}(z-x_{N,j})  \right]^{-\sigma_3},
\label{H}
\end{align}
where $\sigma_3:=\begin{pmatrix} 1 & 0\\ 0 & -1  \end{pmatrix}$ is a Pauli matrix.
\\\\
A straightforward calculation shows that
$H(z)$ is a solution of the
RHP:
\begin{enumerate}
\item[($\,H_a$)] $H(z)$ is analytic for $z\in \mathbb{C}\setminus X_N$;
\item[($\,H_b$)] at each $x_{N,k}\in X_N$, the second column of $H(z)$ is analytic and
the first column of $H(z)$ has a simple pole where the residue is
$$
\Res\limits_{z=x_{N,k}} H(z)=\lim_{z\rightarrow
x_{N,k}} H(z)
\begin{pmatrix} 0 & 0\  \\ \prod\limits^{N-1}_{j=0\atop j\neq k}(z-x_{N,j}  )^{-2}  & 0\  \end{pmatrix};
$$
\item[($\,H_c$)] as $z\rightarrow \infty$,
$$
H(z)=\left(  I + O\left(\frac1z\right)  \right)
\begin{pmatrix}
z^{n-N}  & 0\\
0 & z^{-n+N}
\end{pmatrix}.
$$
\end{enumerate}

We next consider the second transformation which will remove the poles as well as
transform the discrete RHP into a continuous one.
Let $\delta>0$ be a sufficiently small number, and we define
\begin{equation}
    R(z)
:=  H(z)
    \begin{pmatrix}
    1 & \ 0 \ \   \\\\
    \dfrac{\mp ie^{ \pm N \pi iz  }  \cos(N\pi z) }{   N\pi   \prod_{j=0}^{N-1}( z-x_{N,j} )^2  } & \ 1 \
    \end{pmatrix}
\label{R_H-a}
\end{equation}
for $z\in\Omega_{\pm}$, and
\begin{equation}
R(z):=H(z)
\label{R_H}
\end{equation}
for $z\notin\Omega_{\pm}$,
where the domains $\Omega_{\pm}$ are shown in Figure \ref{Figure_R}.
Let $\Sigma_+$ be the boundary of $\Omega_+$ in the upper half-plane, and $\Sigma_-$ be the mirror image of $\Sigma_+$ in the lower half-plane. For the contour $\Sigma=(0,1)\cup\Sigma_\pm$, see also Figure
\ref{Figure_R}.
\begin{figure}[htbp]
\begin{center}
\includegraphics[scale=1]{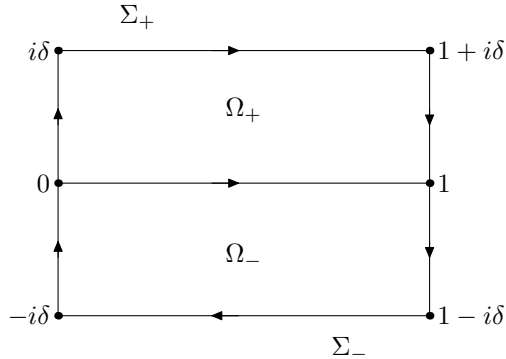}
\caption{The domains $\Omega_{\pm}$ and the contour $\Sigma$.}
\label{Figure_R}
\end{center}
\end{figure}

\begin{lemma}
For each $x_{N,k}\in X_N$,  the singularity of $R(z)$ at $x_{N,k}$ is removable; that is,
$\Res\limits_{z=x_{N,k}}R(z)=0$.
\end{lemma}
\begin{proof}
For  $z\in \Omega_{\pm}$, it follows from (\ref{R_H-a}) that
\begin{equation}
R_{11}(z)=H_{11}(z)+H_{12}(z)
\dfrac{\mp ie^{ \pm N \pi iz  }  \cos(N\pi z) }{   N\pi   \prod_{j=0}^{N-1}( z-x_{N,j} )^2  }
,
\qquad
R_{12}(z)=H_{12}(z)
.
\label{R_11}
\end{equation}
Since $H_{12}(z)$ is  analytic by ($H_b$),
the residue of $R_{12}(z)$ at $x_{N,k}$ is zero.
From ($H_b$), we observe that the residue
of $H_{11}(z)$ at $z=x_{N,k}$ is
\begin{equation}
\Res\limits_{z=x_{N,k}}H_{11}(z)=H_{12}(x_{N,k})\prod\limits^{N-1}_{j=0\atop j\neq k}(x_{N,k}-x_{N,j}  )^{-2}.
\label{H_11}
\end{equation}
On the other hand, it is readily seen that
\begin{equation}
    \Res\limits_{z=x_{N,k}}
    \dfrac{\mp ie^{ \pm N \pi iz  }  \cos(N\pi z) }{   N\pi   \prod_{j=0}^{N-1}( z-x_{N,j} )^2  }
=
    -\prod\limits^{N-1}_{j=0\atop j\neq k}(x_{N,k}-x_{N,j}  )^{-2}.
\label{H_12}
\end{equation}
Thus, applying (\ref{H_11}) and (\ref{H_12}) to (\ref{R_11}) shows
that the residue of $R_{11}(z)$ at $z=x_{N,k}$ is zero.
The entries in the second row of the matrix $R(z)$ can be studied similarly. This completes the proof of the lemma.
\end{proof}

Note that
this transformation makes $R_+(x)$ and $R_-(x)$
continuous on the interval (0,1).
Therefore, $R(z)$ is a solution of the following RHP:
\begin{enumerate}
\item[($R_a$)] $R(z)$ is analytic for $z\in \mathbb{C}\setminus \Sigma$;
\item[($R_b$)] the jump conditions on the contour $\Sigma$: for $x\in (0,1)$,
\begin{equation}
    R_+(x)
=
    R_-(x)
    \begin{pmatrix}
    1 & 0\\
    r(x) & 1
    \end{pmatrix}
,
\label{R_real}
\end{equation}
where
\begin{equation}
r(x)= \frac{4\cos^2(N\pi x)}{ W
\prod\limits_{j=0}^{N-1}(x-x_{N,j})^2 };
\label{r}
\end{equation}
for $z\in \Sigma_{\pm}$,
\begin{equation}
R_+(z)=R_-(z)
\begin{pmatrix}
1 & 0\\
\widetilde{r}_{\pm}(z)  & 1
\end{pmatrix},
\
\end{equation}
where
\begin{equation}
\widetilde{r}_{\pm}(z)=\frac{\mp 2 e^{ \pm N\pi iz }
\cos(N\pi z )}{  W \prod\limits_{j=0}^{N-1}(
z-x_{N,j} )^2  } ;
\label{rtilde}
\end{equation}
\item[($R_c$)] as $z\rightarrow \infty$,
\begin{equation}
R(z)=\left(  I + O\left(\frac1z\right)  \right)
\begin{pmatrix}
z^{n-N}  & 0\\
0 & z^{-n+N}
\end{pmatrix}.
\label{R_inf}
\end{equation}
\end{enumerate}
For simplicity in the following section, here we have introduced the  notation $W:=2N\pi i$.

\section{The Auxiliary Functions}

For our subsequent analysis, we
need  some auxiliary functions.
Consider the monic orthogonal polynomials $p_{N,n}(z):=N^{-n}\pi_{N,n}(Nz-1/2)$.
From (\ref{recurrence relations}) and (\ref{pi_definition}), we have
\begin{equation}
    p_{N,n+1}(z)=(z-a_n)p_{N,n}(z)-b_{n}p_{N,n-1}(z),
\end{equation}
where
$$
    a_n=\frac12, \qquad \qquad b_n=
     \frac{ 1-(n/N)^2 }{ 16-4/n^2 }.
$$


Using the method introduced by  Kuijlaars and  Van Assche \cite{Arno}, we can derive the equilibrium measure corresponding to the discrete Chebyshev polynomials in our case. Recall $c:=n/N$, and let
\begin{equation}
    a:=\frac12 -\frac12 \sqrt{1-c^2},
\qquad \qquad
    b:=\frac12+\frac12 \sqrt{1-c^2}.
\label{MRS}
\end{equation}
The density function is given by
\begin{equation}
  \mu(x)=\left\{\begin{array}{ll}
  \dfrac{2}{\pi c}  \arcsin( \dfrac{c}{2\sqrt{ x-x^2 } } ), &  x\in[a,b],\\
\\
  \dfrac1c,  & x\in[0,a] \cup [b,1].
  \end{array}\right.
\label{density function}
\end{equation}

As usual, we now define the auxiliary $g$-function and $\phi$-function.
\begin{defi}
The $g$-function is the logarithmic potential  defined by
\begin{equation}
g(z):=\int_{0}^{1}  \log(z-s) \mu(s)\,ds, \qquad \qquad z\in \mathbb{C} \setminus (-\infty,1],
\label{g}
\end{equation}
and the so-called $\phi$-function is defined by
\begin{equation}
\phi(z):=\frac l 2 -g(z)
\label{phi}
\end{equation}
for $z\in \mathbb{C} \setminus (-\infty,1]$,
where $l:=2\int_0^1\log|b-s|\mu(s) ds$ is called the Lagrange multiplier.
\end{defi}

An explicit formula for the $g$-function is given in (\ref{gexplicit}) in Section 6. On account of (\ref{density function}), the derivative of $g(z)$
is given by
\begin{equation}
  g'(z)
=
  \frac2c \log\left( \frac{z+\sqrt{ (z-a)(z-b) }+ c/2 }{ (z-1)+\sqrt{ (z-a)(z-b) } -c/2 } \right)
+
  \frac1c \log \left(  \frac{ z-1 }{ z } \right)
.
\label{g1}
\end{equation}
Introduce the ${\phi}^*$-function
\begin{equation}
   \phi^*(z)
:=
    \int_b^{z} \left(-g'(s)\mp \frac{1}{c} \pi i \right) ds
=
    \phi(z)\pm\frac1c \pi i (1-z)
\label{phistar}
\end{equation}
for $z\in \mathbb{C}_{\pm}$.
It is readily seen that
\begin{align}
    {\phi}^*(z)
&=
    -\int_b^{z}
    \left(\frac2c \log\left( \frac{s+\sqrt{ (s-a)(s-b) }+c/2 }{ 1-s-\sqrt{ (s-a)(s-b) } +c/2 } \right)
+
    \frac1c \log \left(  \frac{ 1-s }{ s } \right) \right)
  ds
\label{phistar_integral}
\end{align}
for $z\in \mathbb{C}\backslash (-\infty,b] \cup [1,\infty)$.

Similarly, we also set
\begin{align}
    \widetilde{\phi}(z)
&:=
     \int_a^{z} \left(-g'(s)\mp \frac{1}{c} \pi i \right) ds
=
    \phi\pm \pi i (1-\frac{1}{c}z )
\label{phitilde}
\end{align}
for $z\in \mathbb{C}_{\pm}$.
Note that for $x\in(0,a)$, we have from (\ref{g})
\begin{equation*}
	g_{\pm}(x)
=
	\int_0^1 \log|x-s|\mu(s)ds \pm \pi i \int_x^1 \mu(s)ds.
\end{equation*}
Thus, from (\ref{phitilde}) and (\ref{phi}),
it follows that $\pt_+(x)=\pt_-(x)$ for $x\in(0,a)$; that is,
$\pt(z)$ can be analytically continued to the interval
$(0,a)$.

The functions $\widetilde{\phi}(z)$ and $\phi^*(z)$
play an important role in our argument, and the following are
some of their properties.

\begin{prop}
For $x\in (a,b)$, we have
\begin{equation}
    \phi^*_+(x)+\phi^*_-(x)=0
,\qquad
    \widetilde\phi_+(x)+\widetilde\phi_-(x)=0
.
\label{phi_jump_ab}
\end{equation}
Furthermore,
we have
\begin{equation}
    \phi^*(x)<0
\qquad
    \text{for } x\in (b,1)
\qquad
\text{and}
\qquad
   \widetilde\phi(x)<0
\qquad
    \text{for } x\in (0,a)
.
\label{phi_less_0}
\end{equation}
For any $x\in(b,1)$ and sufficiently small $\delta>0$,
we have
\begin{equation}
     \text{\upshape Re\,}\phi^*(x\pm i\delta )
=
    \phi^*(x)+O(\delta^2)
,\qquad
    \text{\upshape Re\,}\phi(x\pm i\delta )
=
    \phi^*(x)- \frac1c\pi \delta +O(\delta^2)
.
\label{phistar_re}
\end{equation}
For any $x\in(0,a)$ and sufficiently small $\delta>0$,
we have
\begin{equation}
    \text{\upshape Re\,}\widetilde\phi(x\pm i\delta )
=
    \widetilde\phi(x)+O(\delta^2)
,\qquad
    \text{\upshape Re\,}\phi(x\pm i\delta )
=
    \widetilde\phi(x)- \frac1c\pi \delta +O(\delta^2)
.
\label{phitilde_re}
\end{equation}

\end{prop}
\begin{proof}
Since
\begin{equation}
    \phi^*_{\pm}(x)
=
    -\int_b^{x}
    \left(
    \frac2c \log\left( \frac{s\pm i\sqrt{ (s-a)(b-s) }+c/2 }{ 1-s\mp i\sqrt{ (s-a)(b-s) } +c/2 } \right)
+
    \frac1c \log \left(  \frac{ 1-s }{ s } \right)
    \right)
    ds
\end{equation}
for $x\in(a,b)$,
it is easy to verify that
\begin{equation}
    \phi^*_+(x)+\phi^*_-(x)
=
    -\int_b^{x}
    \left(
    \frac2c \log\left( \frac{s+sc }{ (1-s)(1+c)} \right)
+
    \frac2c \log \left(  \frac{ 1-s }{ s } \right)
    \right)
    ds
=
    0
.
\label{e315}
\end{equation}
Similarly, we also have $\widetilde\phi_+(x)+\widetilde\phi_-(x)=0$
for $x\in(a,b)$, thus proving (\ref{phi_jump_ab}).
Together with (\ref{phi}) and (\ref{phistar}), this implies that
$\int_0^1 \log|x-s|\mu(s) ds \equiv l/2$ for $x\in(a,b)$.

For any small $\varepsilon>0$ and
$z\in U(b,\varepsilon):=\{z\in \mathbb{C}: |z-b|<\varepsilon \}$,
we have from (\ref{phistar_integral})
\begin{align}
    \phi^*(z)
&=
     -\int_b^{z}
     \left[
     \frac2c
     \left(
     \frac{ \sqrt{ (b-a)(s-b) } }{ b+c/2 } +
     \frac{ \sqrt{ (b-a)(s-b) } }{ 1-b+c/2 }
     \right)
     +O(\varepsilon)
     \right]
  ds
\nonumber
\\
&=
     -\frac{8 }{3c^2 } (1-c^2)^{1/4} (z-b)^{3/2}
+
    O(\varepsilon^2)
.
\end{align}
Here, we have used the fact that $a+b=1$, $b-a=(1-c^2)^{1/2}$ and
$ab=c^2/4$; see (\ref{MRS}). By straightforward calculation,
it is readily seen from \eqref{phistar_integral} that
\begin{align}
    {\phi^*}'(x)
&=
    -
    \frac2c \log\left( \frac{x+\sqrt{ (x-a)(x-b) }+c/2 }{ 1-x-\sqrt{ (x-a)(x-b) } +c/2 } \right)
-
  \frac1c \log \left(  \frac{ 1-x }{ x } \right)
\notag
\\
&=
	-\frac1c \log
	\left[
		\left(
		\frac{ x+\sqrt{ (x-a)(x-b) }+c/2  }
		{ 1-x-\sqrt{ (x-a)(x-b) } +c/2 }
		\right)^2
		\left( \frac{1-x}{x}  \right)
	\right]
<0
\label{e317}
\end{align}
for $b<x<1$. Thus, we obtain
$\phi^*(x)<\phi(b+\varepsilon)<0$.

In the same manner, if
$z\in U(a,\varepsilon):=\{z\in \mathbb{C}: |z-a|<\varepsilon \}$,
then
\begin{align}
    \widetilde\phi(z)
&=
     \int_a^{z}
    \left[
    \frac2c
    	\left(
    	\frac{ \sqrt{ (a-s)(b-a) } }{ a+c/2 }
    	+
    	\frac{ \sqrt{ (a-s)(b-a) } }{ 1-a+c/2 }
    	\right)
    	+O(\varepsilon)
    \right]
  ds
\nonumber
\\
&=
     -\frac{8 }{3c^2 } (1-c^2)^{1/4} (a-z)^{3/2}
+
    O(\varepsilon^2)
\label{phitilde_experssion}
\end{align}
and
\begin{align}
    \widetilde\phi '(x)
&=
    -
    \frac2c \log\left( \frac{x-\sqrt{ (a-x)(b-x) }+c/2 }{ 1-x+\sqrt{ (a-x)(b-x) } +c/2 } \right)
-
  \frac1c \log \left(  \frac{ 1-x }{ x } \right)
\notag
\\
&=
	-\frac1c \log
	\left[
		\left(
		\frac{ x+\sqrt{ (x-a)(x-b) }+c/2  }
		{ 1-x-\sqrt{ (x-a)(x-b) } +c/2 }
		\right)^2
		\left( \frac{1-x}{x}  \right)
	\right]
>0
\label{e319}
\end{align}
for $0<x<a$.
Consequently,
$\widetilde\phi(x)<\widetilde\phi(a-\varepsilon)<0$ for $0<x<a$,
thus proving (\ref{phi_less_0}).

For any $x\in(b,1)$ and sufficiently small $\delta>0$,
we have from the Taylor expansion
\begin{align}
    \text{\upshape Re\,}\phi^*(x\pm i\delta )
&=
    \text{\upshape Re\,}\phi^*(x )
\mp
    \delta\,\text{\upshape Im\,}{\phi^*}'(x )+O(\delta^2)
=
    \text{\upshape Re\,}\phi^*(x )+O(\delta^2)
\label{phistar_re1}
.
\end{align}
The last equality follows from the fact that $\po{}'(x)$ is real;
see (\ref{e317}).
Coupling this with (\ref{phistar})  gives
\begin{equation}
    \text{\upshape Re\,}\phi(x\pm i\delta )
=
    \text{\upshape Re\,}\phi^*(x )
-
    \frac1c\pi\delta +O(\delta^2)
;
\end{equation}
thus (\ref{phistar_re}) holds.

Similarly, for any $x\in(0,a)$ and sufficiently small $\delta>0$,
we have
\begin{align}
    \text{\upshape Re\,}\widetilde\phi(x\pm i\delta )
&=
    \text{\upshape Re\,}\widetilde\phi(x )
\mp
    \delta\,\text{\upshape Im\,}\widetilde\phi'(x )+O(\delta^2)
=
    \text{\upshape Re\,}\widetilde\phi(x )+O(\delta^2)
.
\label{phitilde_re1}
\end{align}
From (\ref{phitilde}), it follows  that
\begin{equation}
    \text{\upshape Re\,}\phi(x\pm i\delta )
=
    \text{\upshape Re\,}\widetilde\phi(x )
-
    \frac1c\pi\delta +O(\delta^2)
.
\end{equation}
This ends the proof of (\ref{phitilde_re}).
\end{proof}
\begin{figure}[!t]
\begin{minipage}[t]{0.45\linewidth}
\centering
\includegraphics[scale=0.9]{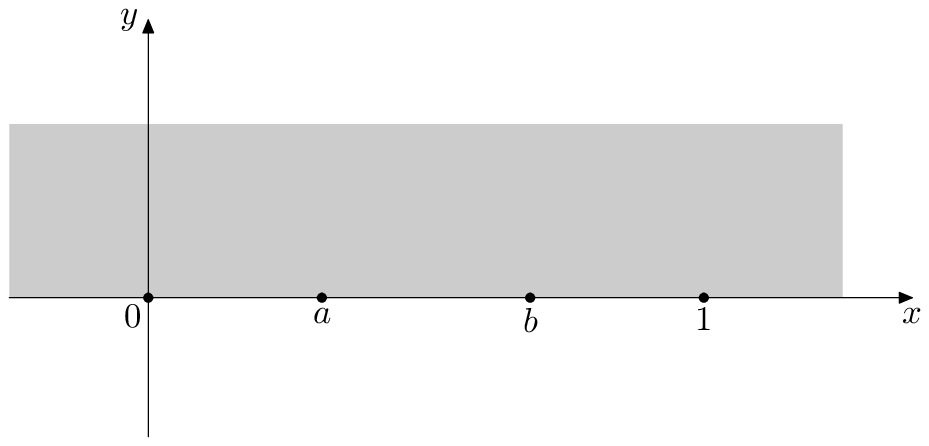}
\\
(a) $z$-plane
\end{minipage}
\hfill
\begin{minipage}[t]{0.45\linewidth}
\centering
\includegraphics[scale=1]{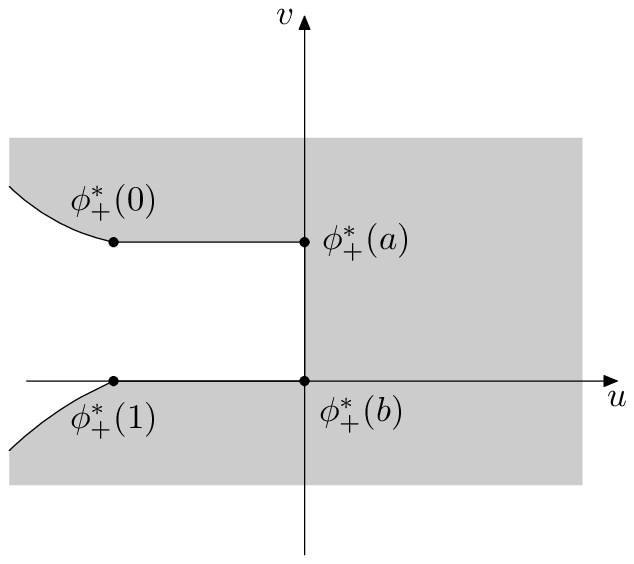}
\\
(b) $\po$-plane
\end{minipage}
\\
\\
\\
\begin{minipage}[t]{0.45\linewidth}
\centering
\includegraphics[scale=1]{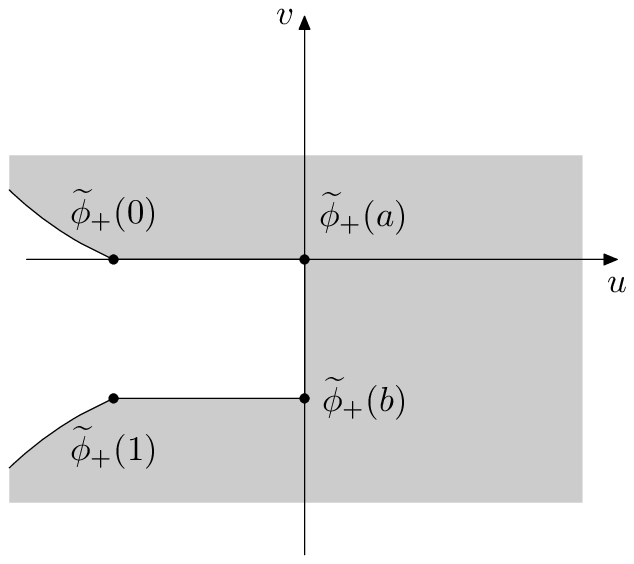}
\\
(c) $\pt$-plane
\end{minipage}
\hfill
\begin{minipage}[t]{0.45\linewidth}
\centering
\includegraphics[scale=1]{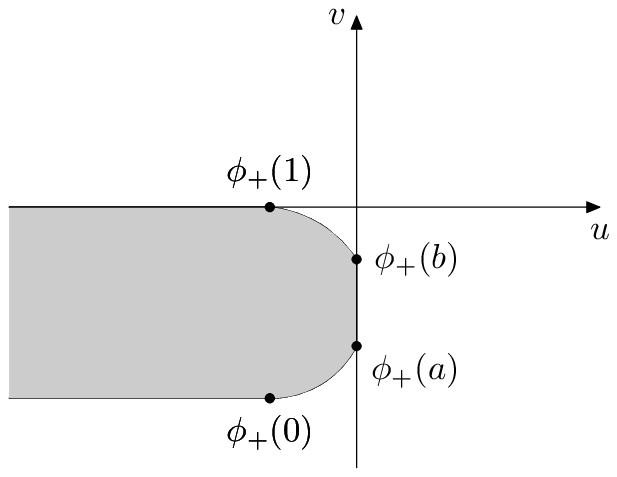}
\\
(d) $\phi$-plane
\end{minipage}
\caption{The upper half-plane under the transformations $\phi$, $\widetilde\phi$ and $\phi^*$.}
\label{Figure_phi}
\end{figure}

From the definition in (\ref{phitilde}),
it is easy to see that $\widetilde{\phi}(a)=0$
and $\widetilde{\phi}_+(b)=(1-\frac1c)\pi i $.
Furthermore, from (\ref{phistar_integral}) it follows that
$\po(b)=0$, and from (\ref{phistar}) and (\ref{phitilde}) we
have $\po_+(a)= (\frac1c -1)\pi i$. Let us now give a brief outline of
the argument used in establishing the following mapping properties of the $\phi$-function.

\begin{prop}
The images of the upper half of the $z$-plane under the mappings $\phi^*$, $\pt$ and $\phi$ are depicted in Figure \ref{Figure_phi}.
\end{prop}

\begin{proof}
  For $x\in(b,1)$, $\po(x)$ is real and negative in view of
(\ref{phi_less_0}). Furthermore, (\ref{e317}) implies that
$\po(x)$ is a monotonically decreasing function in $(b,1)$. Since
$\po(b)=0$, $\po(1)$ is negative. For $x\in(a,b)$, we have from
(\ref{phistar}), (\ref{phi}) and (\ref{g})
$$
	\po_+(x)
	=
	\frac{l}{2}
	-\int_0^1 \log|x-s| \mu(s) ds
	+\pi i \int_x^1 \left(\frac1c -\mu(s) \right)ds.
$$
On account of a statement following (\ref{e315}), the first two terms
on the right-hand side cancel. Hence, we obtain
$$
	\po_+(x)=\pi i \int_x^1 \( \frac1c-\mu(s) \) ds
$$
for $x\in(a,b)$. By (\ref{density function}), $\frac1c > \mu(s)$
for $s\in(0,1)$. Thus, Re$\, \po_+(x)=0$ and Im$\, \po_+(x)>0$ for
$x\in(a,b)$. Furthermore, since $(\text{Im}\, \po_+(x) )'=-\pi (\frac1c -\mu(x))<0$, Im$\, \po_+(x)$ is monotonically decreasing in $(a,b)$.
By a similar argument, it can be shown that for $x\in(1,+\infty)$,
Im$\, \po_+(x)$ is negative and decreasing; also Re$\, \po_+(x)<$
Re$\, \po_+(1)$ and Re$\, \po_+(x)$ is decreasing. Again, by
a similar argument, it can be shown that for $x\in(0,a)$,
Re$\, \po_+(0)=\,$Re$\, \po_+(1)$, Im$\, \po_+(0)=\,$Im$\, \po_+(a)=\pi(\frac1c -1)>0$ and Im$\, \po_+(x)=\pi(\frac1c -1)=\,$Im$\, \po_+(a)$.
Finally, one can show that for $x\in(-\infty,0)$, Im$\,\po_+(x)>\pi(\frac1c-1)>0$, Im$\, \po_+(x)$ is decreasing and Re$\, \po_+(x)$ is
increasing.

Coupling (\ref{phistar}) and (\ref{phitilde}), we have
$$
	\pt(z)=\po(z)\pm \pi i (1-\frac1c),
$$
from which it is easy to verify the image of the upper half-plane
under the mapping $\pt$ shown in Figure 2c.


To verify the graph in Figure 2d, we first show that
for $x\in(-\infty, 0)$, Re$\, \phi_+(x)<\,$Re$\,\phi_+(0)<0$
and Im$\, \phi_+(x)=-\pi$. Then, we show that for $x\in(0,1)$,
Im$\, \phi_+(x)<0$ and Im$\, \phi_+(x)$ is an increasing function.
To deal with Re$\,\phi_+(x)$ in $(0,1)$, we consider three intervals
$(0,a)$, $(a,b)$ and $(b,1)$, separately. For
$x\in(b,1)$, we show that $(\text{Re}\,\phi_+(x))'=(\po(x))'<0$.
Thus, Re$\, \phi_+(x)$ is a decreasing function.
For $x\in(a,b)$, we show that Re$\, \phi_+(x)=0$.
For $x\in(0,a)$, we first show Re$\, \phi_+(x)=\pt(x)$.
By (\ref{phi_less_0}) and (\ref{e319}),
Re$\, \phi_+(x)<0$ and $(\text{Re}\,\phi_+(x))'=(\pt(x))'>0$.
Thus, Re$\, \phi_+(x)$ is an increasing function. This ends
the verification of Figure 2d, and completes the proof of
the proposition.
\end{proof}

To construct our global parametrix, we first define the function
\begin{equation}
    N(z)
=
    \begin{pmatrix}
    \
    \dfrac{ \sqrt{z-a}+\sqrt{z-b} }{ 2(z-a)^{1/4}(z-b)^{1/4}  } &
    -i\dfrac{ \sqrt{z-a}-\sqrt{z-b} }{ 2(z-a)^{1/4}(z-b)^{1/4}  } \ \\\\
    \ i\dfrac{  \sqrt{z-a}-\sqrt{z-b} }{ 2(z-a)^{1/4}(z-b)^{1/4}  } &
    \dfrac{ \sqrt{z-a}+\sqrt{z-b} }{ 2(z-a)^{1/4}(z-b)^{1/4}  }\
    \end{pmatrix}.
\label{N}
\end{equation}
It is readily verified that $N(z)$ is analytic in $\mathbb{C} \setminus [a,b]$ and
\begin{align}
    N_+(x)
=
    N_-(x)
    \begin{pmatrix}
    0  & -1 \\
    1 & 0
    \end{pmatrix},
    \qquad \qquad x\in(a,b).
\label{N_jump}
\end{align}

Next, we introduce the Airy parametrix
\begin{equation}
A(z)
:=
    \left\{
        \begin{array}{ll}
            \begin{pmatrix}
            \text{Ai}(z) & \omega^2 \text{Ai}(\omega^2 z) \\
            i \text{Ai}'(z)  & i \omega \text{Ai}'(\omega^2 z)
            \end{pmatrix}
            &  z\in \mathbb{C}_+;\\\\
            \begin{pmatrix}
            \text{Ai}(z) & -\omega \text{Ai}(\omega z) \\
            i \text{Ai}'(z)  & -i \omega^2 \text{Ai}'(\omega z)
            \end{pmatrix}
            &  z\in \mathbb{C}_- ,
        \end{array}
    \right.
\label{A}
\end{equation}
where $\om=e^{2\pi i/3}$. In view of the well-known identity \cite[(9.2.12)]{Handbook}
\begin{align}
    \Ai(z)+\om\Ai(\om z)+\om^2\Ai(\om^2z)=0
,
\label{Ai_relation}
\end{align}
it is clear that
\begin{equation}
    A_+(z)
=
    A_-(z)
    \begin{pmatrix}
    1 & -1 \\
    0 & 1
    \end{pmatrix},
\qquad x\in\mathbb{R}.
\label{A_jump}
\end{equation}
By (\ref{Ai_relation}), we also have
\begin{eqnarray}\label{A2}
    \A(z)\begin{pmatrix}
        1&0 \\
        \pm1&1
      \end{pmatrix}
=
    \begin{cases}
      \begin{pmatrix}
        -\om\Ai(\om z)&\om^2\Ai(\om^2z) \\
        -i\om^2\Ai'(\om z)&i\om\Ai'(\om^2z)
      \end{pmatrix} &z\in\mathbb{C}_+;\\
      \\
      \left(\begin{matrix}
        -\om^2\Ai(\om^2z)&-\om\Ai(\om z) \\
        -i\om\Ai'(\om^2z)&-i\om^2\Ai'(\om z)
      \end{matrix}\right)&z\in\mathbb{C}_-.
    \end{cases}
\end{eqnarray}
Recall the asymptotic expansions of the Airy function and its derivative  \cite[9.7(ii)]{Handbook}
\begin{equation}
    \text{Ai}(z)
\sim
    \frac{z^{-1/4}}{ 2\sqrt{\pi} } e^{ -\frac23 z^{3/2} }
    \sum_{s=0}^{\infty} \frac{ (-1)^s u_s }{ (\frac23 z^{3/2})^s },
\ \ \ \
    \text{Ai}'(z)
\sim
    - \frac{ z^{1/4} }{ 2\sqrt{\pi} } e^{- \frac23 z^{3/2}}
    \sum_{s=0 }^{\infty } \frac{ (-1)^s v_s }{ (\frac23 z^{3/2} )^s }
\label{Ai_asy}
\end{equation}
as $z\rightarrow \infty$ in $|\text{arg }z |<\pi $, where $u_s, v_s$ are constants with $u_0=v_0=1$.
From (\ref{A}) and (\ref{Ai_asy}), we obtain
\begin{equation}
    \A(z)
=
    \frac{ z^{ -\sigma_3/4 } }{ 2\sqrt{\pi}  }
    \begin{pmatrix}
    1 & -i \\
    -i & 1
    \end{pmatrix}
    (I+ O(|z|^{-\frac32})) e^{-\frac23 z^{3/2} \sigma_3}\label{A_asy}
\end{equation}
as $z\rightarrow \infty$ in $|\text{arg }z |<\pi $,
\begin{equation}
\A(z)
    \begin{pmatrix}
    1 & 0 \\
    \pm 1 & 1
    \end{pmatrix}
=
     \frac{ z^{ -\sigma_3/4 } }{ 2\sqrt{\pi}  }
    \begin{pmatrix}
    1 & -i \\
    -i & 1
    \end{pmatrix}
    (I+ O(|z|^{-\frac32})) e^{-\frac23 z^{3/2} \sigma_3}  \label{A_asy-other}
\end{equation}
as $z\rightarrow \infty$ with $\text{arg\,}z\in (\pm\pi/3, \pm\pi]$.
Finally,
we introduce the $D$-functions:
\begin{equation}
    D(z):=  \dfrac{ e^{Nz} \Gamma(Nz+1/2) }{ \sqrt{2\pi} (Nz)^{Nz}  }
\label{D}
\end{equation}
and
\begin{equation}
    \widetilde{D}(z)
:=
    \dfrac{  \sqrt{2\pi} e^{Nz}(-Nz)^{-Nz} }{ \Gamma(-Nz+1/2) },
\label{Dtilde}
\end{equation}
which were first used in \cite{X.S.Wang} to construct global asymptotic formulas without any cut in the complex plane.
For simplicity, we also introduce the notations
\begin{equation}
    D^*(z):=D(1-z)
,\qquad \qquad
    \widetilde D^*(z) := \widetilde D(1-z)
.
\label{D_star}
\end{equation}
By Euler\rq{}s reflection formula, we have
\begin{equation}
    \widetilde{D}(z)=  2\cos(N\pi z )e^{\pm N\pi i  z} D(z)
\label{D-Dtilde}
\end{equation}
for $z\in \mathbb{C}_{\pm}$.
The reason why we consider $D(z)$ here
is to make sure that
the jump matrix $J_S(z)$ in Lemma 4.1
is asymptotically equal to the identity matrix.
As $n\rightarrow \infty$, by applying Stirling's formula to (\ref{D}) and (\ref{Dtilde}), we obtain
\begin{equation}
    D(z)= 1+O(1/n)
\label{D_asy}
\end{equation}
for $|\text{arg\,}z|<\pi$, and
\begin{equation}
    \widetilde{D}(z)=1+O(1/n)
\label{Dtilde_asy}
\end{equation}
for $|\text{arg}(-z)|<\pi$.

\section{Construction of Parametrix}

\begin{figure}[tbp]
\begin{center}
\includegraphics[scale=1]{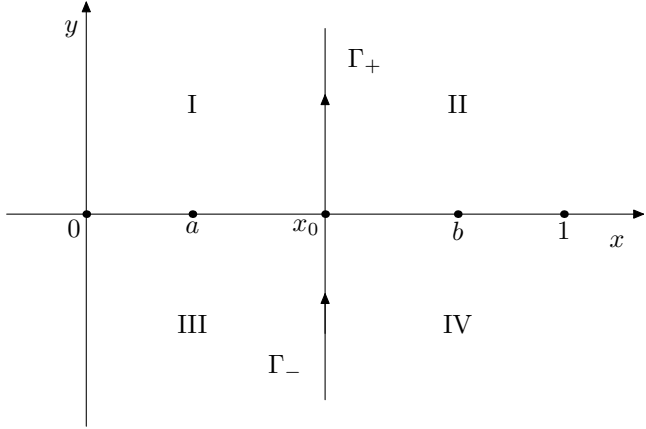}
\caption{The domains I, II, III and IV.}
\label{Figure_Rtilde}
\end{center}
\end{figure}
First, we select an arbitrary fixed point $x_0 \in(a,b)$.
Let $\Gamma$ be the line $\text{Re\,} z= x_0$; see Figure \ref{Figure_Rtilde}.
To facilitate our discussion below, we divide the complex plane into four regions
by using $\Gamma$ and the real line.
With a similar technique as used in \cite{DaiWong} and \cite{ X.S.Wang},
we construct the parametrix of the RHP for $R$ in these four regions.
Define
\begin{equation}
    \widetilde{R}(z):=(-1)^n \sqrt{\pi} (We^{nl} )^{\sigma_3/2} N(z)
        \begin{pmatrix}
        1 & i \\
        i & 1
        \end{pmatrix}
    \widetilde{f}(z)^{-\sigma_3/4}\sigma_1 \A(\widetilde{f})\sigma_1
    \left[ \frac{ 4\cos^2(N\pi z) D(z)^2  }{ W \prod\limits_{j=0}^{N-1}(
    z-x_{N,j} )^2 }  \right]^{\sigma_3/2}
\label{Rtilde}
\end{equation}
for $z\in $ I and III,
\begin{align}
    \widetilde{R}(z):=
    (-1)^N \sqrt{\pi} (We^{nl} )^{\sigma_3/2} N(z)
        \begin{pmatrix}
        1 & -i \\
        -i & 1
        \end{pmatrix}
    &f^*(z)^{-\sigma_3/4}\sigma_2 \A(f^*)\sigma_2^{-1}
\nonumber
\\
    &\qquad \times
    \left[ \frac{ 4\cos^2(N\pi z) D^*(z)^2  }{ W \prod\limits_{j=0}^{N-1}(
    z-x_{N,j} )^2  }  \right]^{\sigma_3/2}
\label{Rstar}
\end{align}
for $z\in $ II and IV,
where $\widetilde{f}(z)$ and $f^*(z)$ are given by
\begin{equation}
    \widetilde{f}(z):= \left(  -\frac32 n\widetilde{\phi}(z)  \right)^{2/3},  \qquad \qquad
    f^*(z) :=\left(  -\frac32 n{\phi}^*(z)  \right)^{2/3}
\label{f}
\end{equation}
and
$
\sigma_1:=
\begin{pmatrix}
0 & 1\\
1 & 0
\end{pmatrix},
\sigma_2:=
\begin{pmatrix}
0 & 1\\
-1 & 0
\end{pmatrix}.
$
To show that $\widetilde{R}(z)$ has the same behavior as $R(z)$ as $z\rightarrow \infty$,
we only consider the case when $z\in$ I.
The other cases can be handled in a similar manner.
A combination of (\ref{A_asy}) and (\ref{f}) gives
\begin{equation}
    \sqrt{\pi}
    \begin{pmatrix}
    1 & i \\
    i & 1
    \end{pmatrix}
    \widetilde{f}(z)^{-\sigma_3/4}\sigma_1 \text{A}(\widetilde{f})\sigma_1
\sim
    (-1)^ne^{-n(l/2-g) \sigma_3 }
    e^{N\pi iz \sigma_3}.
\end{equation}
Here, we have  made use of (\ref{phi}) and (\ref{phitilde}).
Thus, we obtain from (\ref{D-Dtilde}) and (\ref{Rtilde})
\begin{equation}
    \widetilde{R}(z)
\sim
   (We^{nl} )^{\sigma_3/2} N(z)e^{-n(l/2-g)\sigma_3}
   \widetilde{D}^{\sigma_3}
   W^{-\sigma_3/2} \prod\limits_{j=0}^{N-1}(z-x_{N,j} )^{-\sigma_3}
.
\end{equation}
Also note that $\widetilde{D}(z) \sim 1$ and
$\prod_{j=0}^{N-1}(z-x_{N,j} ) \sim z^N$
as $z\rightarrow \infty$.
From (\ref{g}) and the asymptotic behavior of $N(z)$, it
follows that
\begin{align}
    \widetilde{R}(z)
\sim
   (We^{nl} )^{\sigma_3/2} N(z)
   (We^{nl} )^{-\sigma_3/2}e^{ng\sigma_3}z^{-N\sigma_3}
=
    (I+O(1/z))z^{(n-N)\sigma_3}
\label{Rtilde_infty}
\end{align}
as $z\rightarrow\infty$.
Define the matrix
\begin{equation}
    S(z):=(W e^{nl})^{-\sigma_3/2}R(z) \widetilde{R}(z)^{-1} (We^{nl })^{\sigma_3/2}
.\label{S}
\end{equation}
It is easy to see that
\begin{equation}
    S(z)=I+O(1/z)
\end{equation}
as $z\rightarrow \infty$.
The jump matrix of $S(z)$ is given by
\begin{equation}
J_{S}(z):=S_-(z)^{-1}S_+(z)=
(W e^{nl})^{-\sigma_3/2} \widetilde{R}_-(z) J_{R}(z)\widetilde{R}_+(z)^{-1} (W e^{nl})^{\sigma_3/2},
\label{S_jump}
\end{equation}
and the contour associated with the matrix $J_S(z)$ is
$\Sigma_S=\Sigma+\Gamma+\mathbb{R}$; see Figures 1 and 3.
\begin{lemma}
$J_S(z)=I+O(\frac1n) $ and $S(z) =I+O\left(\frac1n\right) $ as $n\rightarrow \infty.$
\end{lemma}

\begin{proof}

For $z\in\Gamma$ and $\pm\,$Im\,$z\in(0,\delta)$,
we have $J_R(z)=I$.
This together with
(\ref{Rtilde}), (\ref{Rstar}) and (\ref{S_jump})
gives
\begin{align}
    J_S(z)
&=
    (-1)^N  N(z)
    \left[
        \sqrt{\pi}
        \begin{pmatrix}
        1 & -i \\
        -i & 1
        \end{pmatrix}
        f^*(z)^{-\sigma_3/4}\sigma_2 A(f^*)\sigma_2^{-1}
    \right]
    (D^*(z)/D(z))^{\sigma_3}
\nonumber
\\
&\quad \times
    \left[
        \sqrt{\pi}
        \begin{pmatrix}
        1 & i \\
        i & 1
        \end{pmatrix}
        \widetilde{f}(z)^{-\sigma_3/4}\sigma_1 A(\widetilde{f})\sigma_1
    \right]^{-1} N(z)^{-1} (-1)^n\
.
\label{S_x0}
\end{align}
On account of (\ref{f}), we have from (\ref{A_asy-other})
$$
    \sqrt{\pi}
    \begin{pmatrix} 1 & -i \\ -i & 1 \end{pmatrix}
    f^*(z)^{-\sigma_3/4}\sigma_2 A(f^*)\sigma_2^{-1}
=
    (I+O(1/n))
    e^{-n{\phi}^* \sigma_3}
    \begin{pmatrix}1 & \pm 1 \\0 & 1\end{pmatrix}
$$
and
$$
    \left[
        \sqrt{\pi}
        \begin{pmatrix}1 & i \\i & 1\end{pmatrix}
        \widetilde{f}(z)^{-\sigma_3/4}\sigma_1 A(\widetilde{f})\sigma_1
    \right]^{-1}
=
    \begin{pmatrix}1 & \mp 1\\0 & 1\end{pmatrix}
    e^{n\widetilde{\phi}\sigma_3}(I+O(1/n))
$$
as $n\rightarrow \infty$.
Applying the above two equations to (\ref{S_x0}) gives
\begin{align*}
J_S(z)
&=
     N(z)
    \begin{pmatrix}
    \frac{D^*(z)}{D(z)}  &  \left(\mp\frac{D^*(z)}{D(z)}\pm\frac{D(z)}{D^*(z)}  \right)e^{-2n\widetilde{\phi} }\
     \\\\
    0 &  \frac{D(z)}{D^*(z)}
    \end{pmatrix}
    N(z)^{-1}(I+O(1/n))
\\
&=
    I+O(1/n)
.
\end{align*}
Here, we have used the fact that
$D(z)=1+O(1/n)$, $D^*(z)=1+O(1/n)$
and Re$\, \widetilde{\phi}> 0$; see Figure \ref{Figure_phi}c.
(A corresponding drawing can be given for the image of the lower
half-plane under the mapping $\pt$.)

For $z\in \Gamma$ and $\pm\,$Im\,$z\notin(0,\delta)$, we have from (\ref{A_asy}) and (\ref{f})
$$
    \sqrt{\pi}
    \begin{pmatrix}
    1 & -i \\
    -i & 1
    \end{pmatrix}
    f^*(z)^{-\sigma_3/4}\sigma_2 A(f^*)\sigma_2^{-1}
=
    e^{-n{\phi}^*\sigma_3}(I+O(1/n))
$$
and
$$
    \left[
        \sqrt{\pi}
        \begin{pmatrix}
        1 & i \\
        i & 1
        \end{pmatrix}
        \widetilde{f}(z)^{-\sigma_3/4}\sigma_1 A(\widetilde{f})\sigma_1
    \right]^{-1}
=
    e^{n\widetilde{\phi}\,\sigma_3}(I+O(1/n)).
$$
Applying the last two equations to (\ref{S_x0}) yields
\begin{align*}
    J_S(z)
&=
    (-1)^{N-n}  N(z)
    e^{-n{\phi}^* \sigma_3}
    \left( \frac{D^*(z)}{D(z)} \right)^{\sigma_3}
    e^{n\widetilde{\phi}\sigma_3}
    N(z)^{-1}(I+O(1/n))
\\
&=
    I+O(1/n)
.
\end{align*}

Let us now consider $z$ in the left half-plane of $\Gamma$.
For $z=x\in (-\infty,0)$, we have $J_R(x)=I$. Coupling (\ref{Rtilde}) and (\ref{S_jump}) gives
\begin{align}
    J_S(z)
&=
    (-1)^n N(z)
    \left[
        \begin{pmatrix}
        1 & i \\
        i & 1
        \end{pmatrix}
        \widetilde{f}_-(z)^{-\sigma_3/4}\sigma_1 A(\widetilde{f}_-)\sigma_1
    \right]
    \left(  D_-(z)/D_+(z)  \right)^{\sigma_3}
\nonumber
\\
&\quad\times
    \left[
        \begin{pmatrix}
        1 & i \\
        i & 1
        \end{pmatrix}
        \widetilde{f}_+(z)^{-\sigma_3/4}\sigma_1 A(\widetilde{f}_+)\sigma_1
    \right]^{-1}
    N(z)^{-1}
    (-1)^n
.
\label{S_0}
\end{align}
Note that arg$\, \widetilde f_{\pm}(x)\in (-\pi,\pi)$ in this
case; see Figure 2c. Hence,
we obtain from (\ref{A_asy}) and (\ref{f})
$$
    \begin{pmatrix}
    1 & i \\
    i & 1
    \end{pmatrix}
    \widetilde{f}_{\pm}(z)^{-\sigma_3/4}\sigma_1 A(\widetilde{f}_{\pm})\sigma_1
=
    \frac{1  }{ \sqrt{ \pi} }e^{-n\widetilde{\phi}_{\pm} \sigma_3 }(I+O(1/n))
.
$$
Substituting the above  equation in (\ref{S_0}) yields
\begin{align*}
J_S(z)
&=
N(z) e^{-n\widetilde{\phi}_- \sigma_3}  e^{2N \pi iz \sigma_3 } e^{ n\widetilde{\phi}_+ \sigma_3} N(z)^{-1}
(I+O(1/n))
\\
&= I+O(1/n).
\end{align*}
Here, we have used the facts
$D_-/D_+= e^{2N \pi i   z}$ and $e^{ n(\widetilde{\phi}_+-\widetilde{\phi}_-) } = e^{-2N \pi i  z }$.

For $z=x\in [0,x_0 ]$,
we have from (\ref{R_real}) and (\ref{r})
$$
J_R(x)
=
    \begin{pmatrix}
    1 & 0 \  \\\\
    \dfrac{4\cos^2(N\pi x )}{ W
    \prod\limits_{j=0}^{N-1}(x-x_{N,j})^2 } & 1 \
    \end{pmatrix}.
$$
Note that $\widetilde{f}(z) \in \mathbb{C}_{\mp}$ when $z\in \mathbb{C}_{\pm}$.
This together with (\ref{Rtilde}) and (\ref{S_jump}) gives
\begin{align}
J_S(x)
&=
   \left[ N(x)  \begin{pmatrix} 1 & i \\ i & 1 \end{pmatrix} \widetilde{f}(x)^{-\sigma_3/4} \sigma_1 A(\widetilde{f}_-) \sigma_1 \right]
       \begin{pmatrix}
       1 & 0 \\
       D^{-2} & 1
       \end{pmatrix}
\nonumber
\\
&\quad \times
   \left[ N(x)  \begin{pmatrix} 1 & i \\ i & 1 \end{pmatrix} \widetilde{f}(x)^{-\sigma_3/4} \sigma_1 A(\widetilde{f}_+) \sigma_1 \right]^{-1}
\end{align}
From (\ref{A_jump}), it follows that
\begin{align}
J_S(x)
&=
   \left[ N(x)  \begin{pmatrix} 1 & i \\ i & 1 \end{pmatrix} \widetilde{f}(x)^{-\sigma_3/4} \sigma_1 A(\widetilde{f}_+) \sigma_1 \right]
   \left[ \sigma_1 \begin{pmatrix} 1 & -1 \\ 0 & 1  \end{pmatrix} \sigma_1  \right]
       \begin{pmatrix}
       1 & 0 \\
       D^{-2} & 1
       \end{pmatrix}
\nonumber
\\
&\quad \times
   \left[ N(x)  \begin{pmatrix} 1 & i \\ i & 1 \end{pmatrix} \widetilde{f}(x)^{-\sigma_3/4} \sigma_1 A(\widetilde{f}_+) \sigma_1 \right]^{-1}
        \nonumber
        \\
&=
    \left[ N(x)  \begin{pmatrix} 1 & i \\ i & 1 \end{pmatrix} \widetilde{f}(x)^{-\sigma_3/4} \sigma_1 A(\widetilde{f}_+) \sigma_1 \right]
    \begin{pmatrix}
    1 & 0 \\
    D^{-2}-1 & 1
    \end{pmatrix}
\nonumber
\\
&\quad\times
    \left[ N(x)  \begin{pmatrix} 1 & i \\ i & 1 \end{pmatrix} \widetilde{f}(x)^{-\sigma_3/4} \sigma_1 A(\widetilde{f}_+) \sigma_1 \right]^{-1}
.
\label{S_0-x_0}
\end{align}
Note that $N(z)$ and $\widetilde{f}(z)^{-1/4}$ are
both discontinues on the interval $(a,x_0]$.
From Figure \ref{Figure_phi}c and (\ref{f}), we observe
that arg$\, \widetilde{\phi}_{\pm}(x)=\mp \pi/2$
and arg$\, \widetilde{f}_{\pm}(x)=\mp \pi$ for $x\in(a,x_0]$.
Thus, we have
$\widetilde{f}_+(x)^{ -\sigma_3/4 }= \widetilde{f}_-(x)^{-\sigma_3/4 } e^{ i \pi \sigma_3/2 }$.
By (\ref{N_jump}), it is readily seen that
$$
    N(z)
    \begin{pmatrix}
    1 & i \\
    i & 1
    \end{pmatrix}
    \widetilde{f}(z)^{-\sigma_3/4}
$$
has no jump on the interval $(a,x_0]$. Applying (\ref{D_asy}) to (\ref{S_0-x_0}) gives $J_S(x)=I+O(1/n)$.

For $z\in \Sigma_{\pm}$ and Re$\, z<x_0$, the jump matrix $J_R(z)$ is
$$
J_R(z)
=   \begin{pmatrix}
    1 & 0 \ \\
    \\
    \dfrac{\mp 2 e^{ \pm i\pi Nz }
    \cos(N\pi z) }{  W \prod\limits_{j=0}^{N-1}(
    z-x_{N,j} )^2  } & 1
    \end{pmatrix};
$$
see (2.12).
This together with (\ref{Rtilde})  gives
\begin{align}
J_S(z)
&=
    N(z)
    \left[
        \begin{pmatrix}
        1 & i \\
        i & 1
        \end{pmatrix}
        \widetilde{f}(z)^{-\sigma_3/4}\sigma_1 A(\widetilde{f})\sigma_1
    \right ]
    \begin{pmatrix}
    1 & 0 \ \ \\
    \\
    \dfrac{\mp e^{\pm i\pi Nz} }{ 2\cos(N\pi z) D^2 } & 1
    \end{pmatrix}
\nonumber
\\
&\quad  \times
    \left[
        \begin{pmatrix}1 & i \\i & 1\end{pmatrix}
        \widetilde{f}(z)^{-\sigma_3/4}\sigma_1 A(\widetilde{f})\sigma_1
    \right]^{-1}
    N^{-1}(z).
\label{S_l-a}
\end{align}
From (\ref{A_asy}) and(\ref{f}), we then obtain
$$
    \begin{pmatrix}
    1 & i \\
    i & 1
    \end{pmatrix}
    \widetilde{f}(z)^{-\sigma_3/4}\sigma_1 A(\widetilde{f})\sigma_1
=
    \frac{e^{-n\widetilde{\phi} \sigma_3}}{ \sqrt{\pi} }(I+O(1/n)).
$$
Applying the above formula to (\ref{S_l-a}) yields
\begin{align*}
    J_S(z)
&=
    N(z)
    \begin{pmatrix}
    1 & 0 \ \ \\\\
    \dfrac{ \mp e^{ 2n\widetilde{\phi} } e^{\pm 2i\pi Nz } }
    { 2\cos(N\pi z) e^{\pm i\pi N z}D(z)^2 }  & 1
    \end{pmatrix}
    N(z)^{-1}
    (I+O(1/n))
\\
&=
    N(z)
    \begin{pmatrix}
    1 & 0 \ \ \\\\
    \dfrac{ \mp e^{2n\phi} }{ \widetilde{D}(z)D(z) }  & 1
    \end{pmatrix}
    N(z)^{-1}
    (I+O(1/n))
.
\end{align*}
One can show that
$\text{Re}\,\phi <0$ in this case;
see Figure \ref{Figure_phi}d. Thus, we have $J_S(z)=I+O(1/n)$
as $n\rightarrow \infty$.

In a similar manner, we can prove that for $z$ in the right half-plane of $\Gamma$, the corresponding jump matrix $J_S(z)$ on $\Sigma$
and on the real line
tends to the identity matrix.
Hence, we have shown that
$J_S(z)=I+O(1/n)$ as $n\rightarrow\infty$
on the contour $\Sigma_S$.
By the main result in \cite{QiuWong}, we conclude that $S(z)=I+O(\frac1n)$ as $n\rightarrow \infty$.

\end{proof}

\section{Main Results}

\begin{theorem}
Let $l$, $D(z)$ and $\widetilde{f}(z)$ be defined  as in (\ref{phi}), (\ref{D}) and (\ref{f}), respectively.
We have
\begin{align}
    \pi_{N,n}(Nz-1/2)
&=
    (-N)^n\sqrt{\pi} e^{nl/2}
\nonumber
\\
&\ \times
    \Bigg\{
        \left[
        \sin(N\pi z) \text{\upshape Ai}(\widetilde{f}(z))+\cos(N\pi z)D(z) \text{\upshape Bi}(\widetilde{f}(z))
        \right]
        \widetilde{A}(z,n)
\nonumber
\\
&\qquad\ +
        \left[
        \sin(N\pi z) \text{\upshape Ai}'(\widetilde{f}(z))+\cos(N\pi z)D(z) \text{\upshape Bi}'(\widetilde{f}(z))
        \right]
        \widetilde{B}(z,n)
    \Bigg\}
\label{pi_left}
\end{align}
for $\, z \in$ {\upshape I and III},
where
$$
    \widetilde{A}(z,n)=\frac{(z-b)^{1/4} }{ (z-a)^{1/4}   }\widetilde{f}(z)^{1/4}\left[1+O(1/n)\right],
    \qquad
    \widetilde{B}(z,n)=\frac{ (z-a)^{1/4} }{ (z-b)^{1/4}  }\widetilde{f}(z)^{-1/4}[1+O(1/n)].
$$
Similarly, with $D^*(z)$ and $f^*(z)$ defined in (\ref{D_star}) and (\ref{f}),
\begin{align}
    \pi_{N,n}(Nz-1/2)
&=
    (-1)^N N^n\sqrt{\pi} e^{nl/2}
\nonumber
\\
&\ \times
    \Bigg\{
        \Big[
        \cos(N\pi z)D^*(z) \text{\upshape Bi}(f^*(z))-\sin(N\pi z) \text{\upshape Ai}(f^*(z))
        \Big]
        {A}^*(z,n)
\nonumber
\\
&\qquad  +
        \Big[
        \cos(N\pi z)D^*(z) \text{\upshape Bi}'(f^*(z))-\sin(N\pi z) \text{\upshape Ai}'(f^*(z))
        \Big]
        {B}^*(z,n)
    \Bigg\}
\label{pi_right}
\end{align}
for $\, z \in$ {\upshape II and IV},
where
$$
    {A}^*(z,n)=\frac{(z-a)^{1/4} }{ (z-b)^{1/4}   }f^*(z)^{1/4}[1+O(1/n)],
    \qquad
    {B}^*(z,n)=\frac{ (z-b)^{1/4} }{ (z-a)^{1/4}  }f^*(z)^{-1/4}[1+O(1/n)].
$$
\end{theorem}

\begin{proof}
From the definition of $S(z)$ in (\ref{S}), we have
$$
    R(z)=(W e^{nl})^{\sigma_3/2} S(z) (We^{nl })^{-\sigma_3/2}\widetilde{R}(z).
$$
For any matrix $X$, we denote its $(i,j)$ element  by $X_{ij}$.
The above formula then gives
$$
    R_{11}(z)=S_{11}(z)\widetilde{R}_{11}(z)
    +S_{12}(z)\widetilde{R}_{21}(z)W e^{nl}
$$
and
$$
    R_{12}(z)=S_{11}(z)\widetilde{R}_{12}(z)
    +S_{12}(z)\widetilde{R}_{22}(z)W e^{nl}
.
$$

First, let us
restrict $z$ to the region I indicated in Figure \ref{Figure_Rtilde}.
A combination of (\ref{N}) , (\ref{A}) and (\ref{Rtilde}) gives
$$
    R_{11}(z)
=
    (-1)^n\sqrt{\pi}e^{nl/2}
    \left[
        -i\omega\text{Ai}(\omega \widetilde{f})
        \widetilde{a}(z,n)
        -i\omega^2 \text{Ai}'(\omega \widetilde{f})
        \widetilde{b}(z,n)
    \right]
    \times
        \frac{ 2\cos(N\pi z) D(z)  }{ \prod_{j=0}^{N-1}(
        z-x_{N,j} ) }
$$
and
$$
    R_{12}(z)
=
    (-1)^n\sqrt{\pi}e^{nl/2}
    \left[
        i\text{Ai}(\widetilde{f})
        \widetilde{a}(z,n)
        +i\text{Ai}'(\widetilde{f})
        \widetilde{b}(z,n)
    \right]
    \times
        \frac{ W \prod_{j=0}^{N-1}(
        z-x_{N,j} ) }{ 2\cos(N\pi z) D(z)  }
,
$$
where
$$
    \widetilde{a}(z,n)=\frac{(z-b)^{1/4} }{ (z-a)^{1/4}   }\widetilde{f}(z)^{1/4}(S_{11}(z)-iS_{12}(z))
$$
and
$$
     \widetilde{b}(z,n)=\frac{ (z-a)^{1/4} }{ (z-b)^{1/4}  }\widetilde{f}(z)^{-1/4}(S_{11}(z)+iS_{12}(z)).
$$

From (\ref{R_H-a}) and (\ref{R_H}), we know that
$H_{11}(z)$ has different expressions in different parts of region I.
Let us first consider the regions I$\,\cap\, \Omega_+ $ and III$\,\cap\, \Omega_-$.
For $z\in$ I$\,\,\cap\,\Omega_+ $,
we have from (\ref{R_H-a})
\begin{align}
    H_{11}(z)
&=
    R_{11}(z)
    -
    \dfrac{\mp ie^{ \pm N \pi iz  }  \cos(N\pi z) }{   N\pi   \prod_{j=0}^{N-1}( z-x_{N,j} )^2  }R_{12}(z)
\nonumber
\\
&=
    (-1)^n\sqrt{\pi} e^{nl/2}\prod_{j=0}^{N-1}( z-x_{N,j} )^{-1}
\nonumber
\\
&\quad\times
    \Bigg\{
        \cos(N\pi z)D(z)
        \left[
            -2i\omega\text{Ai}(\omega\widetilde{f})
            \,\widetilde{a}(z,n)
            -2i\omega^2
            \text{Ai}'(\omega\widetilde{f})
            \,\widetilde{b}(z,n)
        \right]
\nonumber
\\
&\qquad \quad
        -i e^{ i\pi N z}D(z)^{-1}
        \left[
            \text{Ai}(\widetilde{f})
            \,\widetilde{a}(z,n)
            +
            \text{Ai}'(\widetilde{f})
            \,\widetilde{b}(z,n)
        \right]
    \Bigg\}
.
\label{H_1p}
\end{align}
The terms in curly brackets in (\ref{H_1p}) can be rewritten as
\begin{align}
   \cos(N\pi z)&D(z)
        \Bigg[
            \left(-2i\omega\text{Ai}(\omega\widetilde{f})
	    -i\text{Ai}(\widetilde{f})
            -i(D(z)^{-2}-1)\text{Ai}(\widetilde{f})
            \right)
            \,\widetilde{a}(z,n)
\nonumber
\\
&\quad\qquad +
            \left(
            -2i\omega^2
            \text{Ai}'(\omega\widetilde{f})
            -i\text{Ai}'(\widetilde{f})	
            -i(D(z)^{-2}-1)\text{Ai}'(\widetilde{f})
            \right)
            \,\widetilde{b}(z,n)
        \Bigg]
\nonumber
\\
&\qquad
        +\sin(N\pi z)D(z)^{-1}
        \Bigg[
            \text{Ai}(\widetilde{f})
            \,\widetilde{a}(z,n)
            +
            \text{Ai}'(\widetilde{f})
            \,\widetilde{b}(z,n)
        \Bigg].
\label{H_112}
\end{align}
Recall the well-known formula of the Airy functions \cite[(9.2.11)]{Handbook}
\begin{eqnarray}
    \text{Bi}(z)
=
    \pm i \left[ 2 e^{\mp \pi i/3} \text{Ai}(\omega^{\pm 1} z)-\text{Ai}(z) \right]
.
\label{Bi}
\end{eqnarray}
Also, note that
$D(z)\sim 1$
for $z\neq 0$
and
$\text{Ai}(\widetilde{f})$
and
$\text{Ai}'(\widetilde{f})$
are exponentially small as $n\rightarrow \infty$ when $z$ is close to
the origin. Hence,
we can always neglect the terms $(D(z)^{-2} -1) \text{Ai}(\widetilde{f})$  and  $(D(z)^{-2} -1) \text{Ai}'(\widetilde{f})$ in (\ref{H_112}).
This together with Lemma 4.1 and (\ref{Bi}) gives
\begin{align}
    H_{11}(z)
&=
    (-1)^n\sqrt{\pi} e^{nl/2}\prod_{j=0}^{N-1}( z-x_{N,j} )^{-1}
\nonumber
\\
&\  \times
    \Bigg\{
        \left[
        \sin(N\pi z) \text{Ai}(\widetilde{f}(z))+\cos(N\pi z)D(z) \text{Bi}(\widetilde{f}(z))
        \right]
        \widetilde{A}(z,n)
\nonumber
\\
&\qquad\  +
        \left[
        \sin(N\pi z) \text{Ai}'(\widetilde{f}(z))+\cos(N\pi z)D(z) \text{Bi}'(\widetilde{f}(z))
        \right]
        \widetilde{B}(z,n)
    \Bigg\}
.
\label{H_1}
\end{align}

Similarly,
for $z\in$ III$\,\cap\,\Omega_- $
\begin{align}
    H_{11}(z)
&=
    (-1)^n\sqrt{\pi} e^{nl/2}\prod_{j=0}^{N-1}( z-x_{N,j} )^{-1}
\\
& \  \times
    \Bigg\{
        \cos(N\pi z)D(z)
        \left[
            2i\omega^2\text{Ai}(\omega^2\widetilde{f})
            \,\widetilde{a}(z,n)
            +2i\omega
            \text{Ai}'(\omega^2\widetilde{f})
            \,\widetilde{b}(z,n)
        \right]
\nonumber
\\
& \qquad
        +i e^{ -i\pi N z}D(z)^{-1}
        \left[
            \text{Ai}(\widetilde{f})
            \,\widetilde{a}(z,n)
            +
            \text{Ai}'(\widetilde{f})
            \,\widetilde{b}(z,n)
        \right]
    \Bigg\}
.
\end{align}
Again, by (\ref{Bi}), we get  exactly the same formula given in (\ref{H_1}).
Using Lemma 4.1, we can obtain the desired result (\ref{pi_left}).
Now, we show that the asymptotic formula of $H_{11}(z)$ in (\ref{H_1})
holds not only for $z$ in $\Omega_\pm$,
but also for $z$ in the whole I and III.
From the relation between
${R}(z)$ and $H(z)$ in (\ref{R_H}), we know that $H_{11}(z)$ =$R_{11}(z)$
for $z\in$ I $ \setminus\Omega_+$. In contrast to the expansion in (\ref{H_1p}), for $z\in$ I $ \setminus\Omega_+$ the term
$$
    (-1)^n\sqrt{\pi} e^{nl/2}
    \prod_{j=0}^{N-1}(z-x_{N,j})^{-1}
    \times
    \Big\{
    -i e^{ i\pi N z}D(z)^{-1}
        \left[
            \text{Ai}(\widetilde{f})
            \,\widetilde{a}(z,n)
            +
            \text{Ai}'(\widetilde{f})
            \,\widetilde{b}(z,n)
        \right]
    \Big\}
$$
does not appear,
since the quantity $e^{i\pi Nz}\Ai(\widetilde f)$ is exponentially
small as $n$ goes to infinity, in comparison with the other term in (\ref{H_1p}). This suggests that the region of validity of the expansion
given in (\ref{H_1}) can be extended to  $z \in\text{I}\cup\text{III}$.
As a consequence, (\ref{pi_left}) holds for $z \in\text{I}\cup\text{III}$.

In a similar manner,  we have
\begin{align*}
      H_{11}(z)
&=
    (-1)^N\sqrt{\pi} e^{nl/2}\prod_{j=0}^{N-1}( z-x_{N,j} )^{-1}
\\
& \  \times
    \Bigg\{
        \cos(N\pi z)D^*(z)
        \left[
            2i\omega^2\text{Ai}(\omega^2f^*)
            \,{a}^*(z,n)
            +2i\omega
            \text{Ai}'(\omega^2f^*)
            \,{b}^*(z,n)
        \right]
\\
& \qquad
        +i e^{ i\pi N z}D^*(z)^{-1}
        \left[
            \text{Ai}(f^*)
            \,{a}^*(z,n)
            +
            \text{Ai}'(f^*)
            \,{b}^*(z,n)
        \right]
    \Bigg\},
\end{align*}
where
$$
    \,{a}^*(z,n)=\frac{(z-a)^{1/4} }{ (z-b)^{1/4}   }f^*(z)^{1/4}(S_{11}+iS_{12})
$$
and
$$
    \,{b}^*(z,n)=\frac{ (z-b)^{1/4} }{ (z-a)^{1/4}  }f^*(z)^{-1/4}(S_{11}-iS_{12})
$$
for $z\in$ II$\,\cap\, \Omega_+$.
Following the same argument as given above, one can obtain (\ref{pi_right})
for $z\in$ II and IV.
This completes the proof of Theorem 5.1.
\end{proof}

\section{Comparison with Earlier Results}
To compare our formulas in Theorem 5.1 with those given in \cite{PanWong},
we first derive from (\ref{pi_left}) two simple asymptotic  formulas for $\pi_{N,n}(Nz-1/2)$ when $z$ is real and less than $a$; cf. (3.2).
\begin{theorem}
Let $l$, $\pt(z)$, $D(z)$ and $\wt D(z)$ be defined  as in (\ref{phi}), (\ref{phitilde}), (\ref{D}) and (\ref{Dtilde}), respectively. We have
\begin{align}
	\pi_{N,n}(Nz-1/2)
 =
	(-N)^n e^{nl/2}
	&\bigg\{
	e^{ -n \pt(z) }D(z)\cos(N\pi z)
	\frac{  (z-a)^{1/2}+(z-b)^{1/2}  }{ (z-a)^{1/4}(z-b)^{1/4}   }
	\left[
	1+O(1/n)
	\right]
 \notag
 \\
 &\
	+O(e^{ n\text{Re}\,\pt(z) })
	\bigg\}
\label{apprPL}
\end{align}
for $0<z\leq a-\delta<a$, and
\begin{align}
	\pi_{N,n}(Nz-1/2)
 =
 	N^n e^{nl/2}
 	\bigg\{
 	 e^{ -n\phi(z) }\wt D(z)
	\frac{  (z-a)^{1/2}+(z-b)^{1/2}  }{ 2(z-a)^{1/4}(z-b)^{1/4}   }
	\left[
	1+O(1/n)
	\right]
	+O(e^{ n\text{Re}\,\pt(z) })
	\bigg\}
	\label{apprPS}	
\end{align}
for $z<0$.
\end{theorem}
\begin{proof}
From the well-known asymptotic expansions of the Airy function
\cite[(9.7.5)-(9.7.8)]{Handbook},
we have
\begin{align}
	&\Ai(\ft(z))\sim \frac{\ft(z)^{-1/4}}{2\sqrt{\pi}}
	 e^{ n\pt(z)  }
	\sum_{k=0}^{\infty}
	\frac{  (-1)^k u_k  }{  (-n\pt(z))^{k}  }
	,
	&&\Ai'(\ft(z))\sim -\frac{\ft(z)^{1/4}}{2\sqrt{\pi}}
	 e^{ n\pt(z)  }
	\sum_{k=0}^{\infty}
	\frac{  (-1)^k v_k  }{  (-n\pt(z))^{k}  }
	,
\label{Aift}
\end{align}
and
\begin{align}
	&\Bi(\ft(z))\sim \frac{\ft(z)^{-1/4}}{\sqrt{\pi}}
	 e^{ -n\pt(z)  }
	\sum_{k=0}^{\infty}
	\frac{  u_k  }{  (-n\pt(z))^{k}  },
	&&	
	\Bi'(\ft(z))\sim \frac{\ft(z)^{1/4}}{\sqrt{\pi}}
	 e^{ -n\pt(z)  }
	\sum_{k=0}^{\infty}
	\frac{  v_k  }{  (-n\pt(z))^{k}  }
,
\label{Bift}
\end{align}
where $\ft(z)$ is defined in (\ref{f}).
Note that $\pt(z)$ is analytic for real $z\in(0,a)$; see the paragraph containing (\ref{phitilde}).
Applying (\ref{Aift})-(\ref{Bift}) to (\ref{pi_left}) gives
\begin{align}
	\pi_{N,n}(Nz-1/2)
=
	(-N)^n  e^{nl/2}
	&\left\{
	e^{n\pt(z)}\sin(N\pi z)
	\frac{  (z-b)^{1/2}-(z-a)^{1/2} }{ 2(z-a)^{1/4}(z-b)^{1/4}   }\left[
	1+O(1/n)
	\right]
	\right.
	\notag
	\\
&\quad  +
	\left.
	e^{-n\pt(z)}D(z)\cos(N\pi z)
	\frac{  (z-a)^{1/2}+(z-b)^{1/2}  }{ (z-a)^{1/4}(z-b)^{1/4}   }
	\left[
	1+O(1/n)
	\right]
	\right\}
\label{apprLd}
 .
\end{align}
Since  $\pt(z)<0$ in this case on account of (\ref{phi_less_0}),
$e^{n\pt(z)}$ is exponentially small as $n$ goes to infinity.
Thus, (\ref{apprPL}) follows from (\ref{apprLd}).

Note that  $\pt(z)$ has a jump across the negative real axis. For real $z<0$,
we obtain
from (\ref{Aift})-(\ref{Bift}) and (\ref{pi_left})
\begin{align}
	\pi_{N,n}^+(Nz-1/2)
=
	(-N)^n  e^{nl/2}
	&\left\{
	e^{n\pt_+(z)}\sin(N\pi z)
	\frac{  (z-b)^{1/2}-(z-a)^{1/2} }{ 2(z-a)^{1/4}(z-b)^{1/4}   }\left[
	1+O(1/n)
	\right]
	\right.
	\notag
	\\
&\ \
 +
	\left.
	e^{-n\pt_+(z)}D_+(z)\cos(N\pi z)
	\frac{  (z-a)^{1/2}+(z-b)^{1/2}  }{ (z-a)^{1/4}(z-b)^{1/4}   }
	\left[
	1+O(1/n)
	\right]
	\right\}
\label{apprD1}
\end{align}
as $n\rightarrow \infty$,
where $\pi_{N,n}^+(Nz-1/2)$ denotes the limiting value of $\pi_{N,n}(Nz-1/2)$ as $z$ approaches the real line from above.
Observe from Figure 2c that Re$\, \pt_+(z)<0$. Thus, it follows from (\ref{apprD1}) that
\begin{align}
	\pi_{N,n}^+(Nz-1/2)
 =
	(-N)^ne^{nl/2}
	&\bigg\{
	e^{ - n\pt_+(z) } D_+(z) \cos(N\pi z)
	\frac{  (z-a)^{1/2}+(z-b)^{1/2}  }{ (z-a)^{1/4}(z-b)^{1/4}   }
	\left[
	1+O(1/n)
	\right]
\notag
\\
&\ \
 	+O(e^{ n\text{Re}\,\pt(z) })
 	\bigg\}
 .
\label{apprD11}
\end{align}
Similarly,  one can see that
 $\pi_{N,n}^-(Nz-1/2)$ is given by
\begin{align}
	(-N)^ne^{nl/2}
	\bigg\{
	 e^{ - n\pt_-(z) } D_-(z)\cos(N\pi z)
	\frac{  (z-a)^{1/2}+(z-b)^{1/2}  }{ (z-a)^{1/4}(z-b)^{1/4}   }
	\left[
	1+O(1/n)
	\right]
 	+O(e^{ n\text{Re}\,\pt(z) })
 	\bigg\}
\label{apprD12}
 .
\end{align}
From the definition of $\phi$ in (\ref{phi}), we note that
$e^{n\phi(z)}$ can be analytically extended to $(-\infty,0)$.
Thus, from (\ref{phitilde}) and (\ref{D-Dtilde}) it follows that
\begin{align}
	e^{- n\pt_\pm(z)  }D_\pm(z)\cos(N\pi z)
	=\frac{  (-1)^n  }{  2  }e^{-n\phi(z) }\wt D(z).
\label{D11}
\end{align}
In view of (\ref{D11}), the two asymptotic formulas (\ref{apprD11}) and (\ref{apprD12}) are exactly the same. Hence,
\begin{align}
  \pi_{N,n}(Nz-1/2)
 &=
 	N^n e^{nl/2}
 	\bigg[
 	 e^{ -n\phi(z) }\wt D(z)
	\frac{  (z-a)^{1/2}+(z-b)^{1/2}  }{ 2(z-a)^{1/4}(z-b)^{1/4}   }
	\left[
	1+O(1/n)
	\right]
	+O(e^{ n\text{Re}\,\pt(z) })
	\bigg]
,
\end{align}
 for real $z<0$.  This gives  (\ref{apprPS}).
\end{proof}

We will now show that our asymptotic formulas for the discrete Chebyshev polynomials are the same as those given by Pan and Wong \cite{PanWong}. First, we introduce the notation
\begin{align}
	x:=Nz-1/2
.
\label{fixedx}
\end{align}
Two different asymptotic approximations for $t_n(x,N+1)$ are given in \cite[(8.13) and (8.6)]{PanWong}, one for $x$ negative
and the other for $x$ positive. Changing $N+1$ to $N$, they read
\begin{align}
	t_n(x,N) &=
	\frac{  (-1)^{n+1}\Gamma(n+N+1)(N-1)^x n^{-2x-2} \Gamma(x+1)  }
	{ \Gamma(N)\pi   }
	\notag
	\\
	&\qquad\times
	\left\{
	\sin(\pi x)
	\left[
	1+ O\(\frac1N\)
	\right]+O(e^{N\eta})
	\right\}
\label{PanL}
\end{align}
for fixed $x>0$ and
\begin{align}
	t_n(x,N) &=
	\frac{  (-1)^{n}\Gamma(n+N+1)(N-1)^x n^{-2x-2}  }
	{ \Gamma(N)\Gamma(-x)   }
	\left[
	1+ O\(\frac1N\)
	\right]
\label{PanS}
\end{align}
for fixed $x<0$.


We next derive from (\ref{apprPL}) and (\ref{apprPS})  asymptotic formulas for $t_{n}(x,N)$ when $x$ is fixed
(i.e., $z=O(1/N)$). For $x>0$, i.e., $z>0$,  substituting (\ref{fixedx}) in (\ref{D}) gives
\begin{align}
	D(z)\cos(N\pi z)=\frac{ e^{Nz} \Gamma(Nz+1/2) }{ \sqrt{2\pi}(Nz)^{Nz} }\cos(N\pi z)=-\frac{ e^{x+1/2} \Gamma(x+1) }{ \sqrt{2\pi}(x+1/2)^{x+1/2} }
\sin(\pi x)	
.
\label{Dx}
\end{align}
Moreover, it is readily verified that
\begin{align}
	 \frac{  (2n)!  }{ n!^2   }
 \sim
	\frac{  1  }{ \sqrt{\pi}    }2^{2n}n^{-1/2}
\label{LGamma}
\end{align}
as $n\rightarrow \infty$ and
\begin{align}
	\frac{  (z-a)^{1/2}+(z-b)^{1/2}  }{ (z-a)^{1/4}(z-b)^{1/4}   }
\sim
   	\frac{a^{1/2}+b^{1//2}  }{ (ab)^{1/4} }
 =
	\frac{(a+b+2a^{1/2}b^{1/2})^{1/2}  }{ (ab)^{1/4} }
 =
	\sqrt{2}(1+c)^{1/2}c^{-1/2}
\label{cosx}
\end{align}
as $z \rightarrow 0$. Here, we have made use of  the fact that
$a+b=1$ and $ab=c^2/4$; see (3.2).

Note that $z\rightarrow 0$ as $n\rightarrow \infty$; a combination of  (2.1),  (\ref{Dx})-(\ref{cosx})  and (\ref{apprPL}) yields
\begin{align}
	t_n(x,N)
 \sim
 	(-1)^{n+1}
 	 N^n  2^{2n}n^{-1/2}
 	\frac{ e^{x+1/2} \Gamma(x+1) }{ \pi (x+1/2)^{x+1/2} }
 	(1+c)^{1/2}c^{-1/2}\sin(\pi x)
 	e^{n(l/2-\pt(z)) }
\label{apprTLD}
\end{align}
as $n\rightarrow \infty$. We now derive an explicit formula for $l/2-\pt(z)$; see (\ref{Pt}) below.
From  (\ref{phi}) and (\ref{phitilde}),  one has
\begin{align}
	l/2-\pt(z)
	=l/2-\phi_+(z)-\pi i(1-\frac1c z)
	=
	\text{Re}\,g_+(z)
\label{ptexplicit}
\end{align}
for $0<z\leq a-\delta$.
This fact evokes us to calculate $g(z)$ first. From (\ref{density function}) and (\ref{g}), it is easily seen that
\begin{align}
	g(z)=\frac1c
	\int_0^a \log(z-s)ds +\frac1c \int_b^1 \log(z-s)ds
	+\int_a^b\log(z-s) \mu(s)ds
.
\end{align}
By integration by parts twice, we obtain
\begin{align}
	\int_a^b  \log(z-s) \mu(s) ds
	&=\mu(s)\log(z-s) s \Big |_{a}^{b}
	-z\mu(s)\log(z-s)  \Big |_{a}^{b}
	-\int_a^b \mu(s) ds
	\notag
        		\\
        		&\qquad
	+z\int_a^b  \mu(s) ' \log(z-s) ds
        		-
       			\int_a^b s  \log(z-s) \mu(s) ' ds
 .
 \label{temp1}
\end{align}
By  the definition of $\mu(x)$, we have
\begin{align}
	\mu(x)'
	=
	-\frac {1}{2\pi} \frac{1-2x}
	{ (x-x^2) \sqrt{ (x-a)(b-x) } }
 .
\end{align}
Here, we have again used the relations $a+b=1$ and $ab=c^2/4$. Moreover,  one can show that
\begin{align}
	\int_a^b  \log(z-s) \mu(s) ds
	&=
	\bigg[
	\frac1c (z-a)\log(z-a)-
 	\frac1c (z-b)\log(z-b)
	\bigg]
	-(1-\frac{2a}{c})
		\notag
        		\\
        		&\quad	
	-\frac{  z  }{ 2\pi   } \int_a^b   \frac{\log(z-s)}
		{ s \sqrt{ (s-a)(b-s) } }ds        	
        		+\frac{  1  }{ 2\pi   } (z-1) \int_a^b   \frac{\log(z-s)}
		{ (1-s) \sqrt{ (s-a)(b-s) } }ds
        		\notag
        		\\
        		&\quad	
        		+
        		\frac {1}{\pi} \int_a^b   \frac{\log(z-s)}
		{  \sqrt{ (s-a)(b-s) } }ds
.
\end{align}
Let the integrals on the right-hand side of the equality be denoted by $I_1$,
$I_2$ and $I_3$, respectively. To evaluate $I_1$,  we note that from \cite{WangWong_Stieltjes-Wigert} we have
\begin{align}
	\int_a^b \frac{  \log s   }{  (s-z) \sqrt{ (s-a)(b-s) } }ds
	=\frac{  2\pi  }{  (z-a)^{1/2}(z-b)^{1/2}  }
	\log \frac{  z+\sqrt{ ab }+(z-a)^{1/2}(z-b)^{1/2}  }
	{ (\sqrt a +\sqrt b)z   }
 .
\label{wangintegral}
\end{align}
Making the change of variable $s=z-x$, (\ref{wangintegral}) gives
\begin{align}
		I_1
	&=
		\int_{z-a}^{z-b}   \frac{\log x }
		{ (z-x)  \sqrt{ (z-x-a)(b-z+x) } }d(z-x)
	\hh
	-\frac{4\pi}{c} \log
		\frac{ z+\sqrt{ (z-a)(z-b) }+c/2  }
		{ (\sqrt{z-a}+\sqrt{z-b})z }.
\end{align}
Similarly, we can evaluate $I_2$ and $I_3$. Thus, a straightforward calculation shows that
\begin{align}
	g(z)
&=
	-1-2\log 2
           +\frac1c (z-1)\log(z-1)-\frac1c z\log z
                         +(2-\frac2c)\log\left[ (z-a)^{1/2}+ ({z-b})^{1/2} \right]
            \notag
            \\
            &\ +
                \frac2c z
                   \log\left[ z+  (z-a)^{1/2}(z-b)^{1/2}+c/2  \right]+\frac2c (1-z)
                \log
        		\left[ z-1+ (z-a)^{1/2}(z-b)^{1/2}-c/2  \right]
.
\label{gexplicit}
\end{align}
By (\ref{ptexplicit}) and (\ref{gexplicit}), $l/2-\pt(z)$  can be explicitly given by
\begin{align}
	&-1-2\log 2
           +\frac1c (z-1)\log(1-z)-\frac1c z\log z
                         +(2-\frac2c)\log\left[ (a-z)^{1/2}+ ({b-z})^{1/2} \right]
            \notag
            \\
            &\quad+
                \frac2c z
                   \log\left[ z-  (a-z)^{1/2}(b-z)^{1/2}+c/2  \right]+\frac2c (1-z)
                \log
        		\left[ 1-z+ (a-z)^{1/2}(b-z)^{1/2}+c/2  \right]
\label{Pt}
.
\end{align}
Note that $n=cN$ and $z\rightarrow 0$ as $n\rightarrow \infty$.
From (\ref{Pt}), we have
\begin{align}
	l/2-\pt(z)=
	\left[
		-1-2\log 2 +(1+\frac1c)\log(1+c)
		\right]
		+z
		\left[
		\frac1c \log(z) -\frac2c \log c -\frac1c
		\right]+O(z^2)
\end{align}
and so
\begin{align}
	e^{n(l/2-\pt(z))}
	\sim
	e^{-n}2^{-2n} (1+c)^{n+N}(x+1/2)^{x+1/2}e^{-x-1/2}N^{x+1/2}n^{-2x-1}
.
\label{apprept}
\end{align}
Substituting (\ref{apprept}) in (\ref{apprTLD}) yields
\begin{align}
	t_n(x,N)
 \sim
	\frac{  (-1)^{n+1}\Gamma(x+1)n^{-2x-2}  }{  \pi }	\sin(\pi x)
	N^{n+x+1} (1+c)^{n+N+1/2} e^{-n}
\label{apprTL}
\end{align}
which will agree with (\ref{PanL}) to leading order if we can show that
\begin{align}
	\frac{  \Gamma(n+N+1)(N-1)^x  }{ \Gamma(N)   }
	\sim N^{n+x+1}(1+c)^{n+N+1/2}e^{-n}.
 \label{PGamma}	
\end{align}
The last result is a direct consequence of Stirling's formula.

Next, we consider the case when $x$ is negative, but for a moment let's restrict
$-\infty<x<-\frac12$ (i.e., $z<0$).
Inserting (\ref{fixedx}) in (\ref{Dtilde}), we have
\begin{align}
	\wt D(z)= \frac{  \sqrt{2\pi}e^{x+1/2}  }{ (-x-1/2)^{x+1/2}\Gamma(-x)   }
.
\label{Dtildex}
\end{align}
A combination of (\ref{apprPS}),  (\ref{LGamma})-(\ref{cosx}) and (\ref{Dtildex}) gives
\begin{align}
	t_n(x,N)
 \sim
 	 N^n  2^{2n}n^{-1/2}
 	\frac{ e^{x+1/2} }{  (-x-1/2)^{x+1/2} \Gamma(-x) }
 	(1+c)^{1/2}c^{-1/2}
 	e^{n(l/2-\phi(z)) }
.
\label{apprTSD}
\end{align}
Since $e^{n\phi(z)}$ is analytic in the interval $(-\infty,0)$, we obtain from
(\ref{phi}) that $e^{n(l/2-\phi(z) )}=e^{n(l/2-\phi_+(z)  )}= e^{ng_+(z)}$.
On account of  (\ref{gexplicit}), it can be shown that
\begin{align}
	g_+(z)=\left[\pi i
		-1-2\log 2 +(1+\frac1c)\log(1+c)
		\right]
		+z
		\left[
		\frac1c \log(-z) -\frac2c \log c -\frac1c
		\right]+O(z^2)
\label{Pp}
\end{align}
as $z\rightarrow 0$.
Thus, letting $n\rightarrow \infty$, we have from (\ref{Pp})
\begin{align}
	e^{n(l/2-\phi(z))}
	\sim
	(-1)^n 2^{-2n}(1+c)^{n+N}e^{-n}
	 c^{-2x-1}N^{-x-1/2} (-x-1/2)^{x+1/2} e^{-x-1/2}
.
\label{t22}
\end{align}
From (\ref{apprTSD}) and  (\ref{t22}),
we obtain
\begin{align}
	t_n(x,N)\sim
	(-N)^n \frac{   n^{-2x-2}  }{   \Gamma(-x) }	
	(1+c)^{n+N+1/2}N^{x+1}e^{-n}
.
 \label{temp1}
\end{align}
In the case
$-\frac12<x<0$, formula (\ref{temp1}) follows directly from (\ref{apprTL}) on an appeal to the
reflection formula $\Gamma(1+x)\Gamma(-x)=-\pi/\sin(\pi x)$.
On the other hand, coupling  (\ref{PanS}) and (\ref{PGamma})  gives
\begin{align}
	t_n(x,N)\sim
	(-N)^n \frac{   n^{-2x-2}  }{   \Gamma(-x) }	
	(1+c)^{n+N+1/2}N^{x+1}e^{-n}
\end{align}
as $n\rightarrow \infty$.
Therefore, our results (\ref{apprTL}) and (\ref{temp1}) agree with those
of Pan and Wong \cite{PanWong} stated in (\ref{PanL}) and (\ref{PanS}).

%
%

\end{document}